\documentclass[12pt]{amsart}
\usepackage{amssymb,amsmath,amsthm,bm}
\usepackage[colorinlistoftodos,prependcaption,textsize=tiny]{todonotes}
 \definecolor{darkblue}{RGB}{0,0,160}
\usepackage{fouriernc}
\usepackage[colorlinks=true,allcolors=darkblue]{hyperref}
\DeclareSymbolFont{usualmathcal}{OMS}{cmsy}{m}{n}
\DeclareSymbolFontAlphabet{\mathcal}{usualmathcal}

\usepackage[T1]{fontenc}
\usepackage{parskip}
\usepackage{setspace}

\usepackage{relsize}
\usepackage[all]{xy}
\usepackage{hyperref}
\usepackage{mathrsfs}
\usepackage{verbatim}
\usepackage[english]{babel}

\usepackage{amsmath,amsthm,amscd,amssymb, mathrsfs}

\usepackage{latexsym}

\numberwithin{equation}{section}

\theoremstyle{plain}
\newtheorem{theorem}{Theorem}[section]
\newtheorem{lemma}[theorem]{Lemma}
\newtheorem{corollary}[theorem]{Corollary}
\newtheorem{proposition}[theorem]{Proposition}

\newtheorem{construction}[theorem]{Construction}

\theoremstyle{definition}
\newtheorem{definition}[theorem]{Definition}

\theoremstyle{remark}

\newtheorem{case[theorem]}{Case}

\title[\parbox{14cm}{\centering{An asymmetric bound for sum of distance sets \hspace{1in}}} \quad]{An asymmetric bound for sum of distance sets}
\author{Daewoong Cheong, Doowon Koh and Thang Pham}

\address{Department of Mathematics\\
Chungbuk National University \\
Cheongju, Chungbuk 28644 Korea}
\email{daewoongc@chungbuk.ac.kr}

\address{Department of Mathematics\\
Chungbuk National University \\
Cheongju, Chungbuk 28644 Korea}
\email{koh131@chungbuk.ac.kr}

\address{ETH Zurich, Switzerland}
\email{vanthang.pham@inf.ethz.ch}


\begin{document}
\begin{abstract}
For $ E\subset \mathbb{F}_q^d$, let $\Delta(E)$ denote the distance set determined by pairs of points in $E$. By using additive energies of sets on a paraboloid, Koh, Pham, Shen, and Vinh (2020) proved that if $E,F\subset \mathbb{F}_q^d $ are subsets with $|E||F|\gg q^{d+\frac{1}{3}}$ then $|\Delta(E)+\Delta(F)|> q/2$. They also proved that the threshold $q^{d+\frac{1}{3}}$ is sharp when $|E|=|F|$. In this paper, we provide an improvement of this result in the unbalanced case, which is essentially sharp in odd dimensions. The most important tool in our proofs is an optimal $L^2$ restriction theorem for the sphere of zero radius.
\end{abstract}
\maketitle
\section{Introduction}
Let $\mathbb{F}_q$ be a finite field of order $q$, where $q$ is an odd prime power. For  $x=(x_1, \ldots, x_d)$ and $(y_1, \ldots, y_d)$ in $E$, define a distance between $x$ and $y$  by
\[||x-y||:=(x_1-y_1)^2+\cdots+(x_d-y_d)^2,\]
which is the square of the Euclidean distance. We denote by $\Delta(E)$ the set of distances determined by pairs of points in $E$, namely,
\[\Delta(E):=\{||x-y||\colon x, y\in E\}.\]
The Erd\H{o}s-Falconer distance problem in $\mathbb{F}_q^d$ asks for the smallest exponent $N$ such that for any $E\subset \mathbb{F}_q^d$ with at least $q^N$ elements, the number of distinct distances determined by pairs of points in $E$ is at least $cq$, for some constant $0<c<1$.

Iosevich and Rudnev \cite{io} showed that if $|E|\ge 4q^{(d+1)/2}$, then $|\Delta(E)|=q$, which means that for any $\lambda\in \mathbb{F}_q$, there exist two points $x, y\in E$ such that $||x-y||=\lambda$.  Hart et al. \cite{hart} proved that the exponent $\frac{d+1}{2}$ is essentially sharp in odd dimensions, even though we wish to cover a positive proportion of all distances. However, in even dimensions, it is conjectured that the right exponent should be $\frac{d}{2}$. We refer the interested reader to \cite{mot, mu} and references therein for most recent progress on this conjecture.

Let $E$ and $F$ be sets in $\mathbb{F}_q^d$, in this paper, we focus on the following analogue: How large do $E$ and $F$ need to be 
to guarantee the inequality 
\[|\Delta(E)+\Delta(F)|\gg q?\]

Here, and throughout this paper, for simplicity, we will use  $C$ to denote  a sufficiently large constant independent of the field size $ q.$ We also use the following notations: $X \ll Y$ means that there exists  some absolute constant $C_1>0$ such that $X \leq C_1Y$, the notation $X \gtrsim Y$ means $X\gg (\log Y)^{-C_2} Y$ for some absolute constant $C_2>0$, and  $X\sim Y$ means $Y\ll X\ll Y$.

One can easily check that $\Delta(E)+\Delta(F)=\Delta(E\times F)$, where $\Delta(E\times F)$ is obviously defined as the distance set determined by the product set $E\times F$ in $\mathbb F_q^{2d}.$ Thus, it follows trivially from Iosevich-Rudnev's result that if $|E||F|\ge4 q^{\frac{2d+1}{2}}$, then $\Delta(E)+\Delta(F)=\mathbb{F}_q$. In a recent paper, by using results on additive energies of sets on a paraboloid, Koh, Pham, Shen, and Vinh \cite{kohmz} showed that the exponent $d+\frac{1}{2}$ can be improved to $d+\frac{1}{3}$ as follows.

\begin{theorem}[\cite{kohmz}]\label{thm1mz}Let $E$ and  $F$ be sets in $\mathbb{F}_q^d$. Suppose either $d=4k-2$ and $q\equiv 3\mod 4$ or $d\ge 3$ is odd. If $|E||F|\ge Cq^{d+\frac{1}{3}}$ for some big positive constant $C$, then
$|\Delta(E)+\Delta(F)|> \frac{q}{2}.$
\end{theorem}

They also constructed examples to show that the exponent $d+\frac{1}{3}$ is optimal for the case $|E|=|F|$ in odd dimensions. Using the same approach, they indicated that the Erd\H{o}s-Falconer distance conjecture holds for sets of the form $A^4\subset \mathbb{F}_q^4$, where $q$ is a prime number and $A$ is a multiplicative subgroup in $\mathbb{F}_q^*.$

We note that in the form of $\Delta(E, F)+\Delta(E, F)$, where $\Delta(E, F):=\{||x-y||\colon x\in E, y\in F\}$, it was first proved by  Shparlinski \cite[Corollary 2]{shparlinski} that 
\begin{equation} \label{Shparlinski} |\Delta(E, F)+\Delta(E, F)|\ge \frac{1}{3}\min\left\lbrace q, \frac{|E||F|}{q^{d-1}}, \frac{|E||F|^2}{q^{\frac{3d}{2}}}\right\rbrace,\end{equation}
which is non-trivial when the sizes of sets $E$ and $F$ differ significantly. A  simple  graph theoretic proof of this result and applications can be found in the work of Hegyv\'{a}ri and  P\'{a}lfy in \cite{HP}.

From the Iosevich-Rudnev result \cite{io}, we observe   that if the size of $E$ (resp. $F$) is at least $4q^{\frac{d+1}{2}}$, then $|\Delta(E)|= q$ (resp. $|\Delta(F)|= q$),  
 and so we have $|\Delta(E)+\Delta(F)|= q$. Hence, in the rest of this paper, without loss of generality, we assume that $|E|, |F|< 4q^{\frac{d+1}{2}}$.

The main purpose of this paper is to give improvements of Theorem \ref{thm1mz} in the unbalanced case, namely, $|E|\ne |F|$. The most important tool in our proofs  is  an optimal $L^2$ restriction theorem for the sphere of zero radius. More precisely, our first result is as follows.

\begin{theorem}\label{thm0}
Let $E$ and $F$ be sets in $\mathbb{F}_q^d$. 
Suppose either $d=4k-2$ and $q\equiv 3\mod 4$ or $d\ge 3$ is odd. Then there exists a large positive constant $C$ such that  if either $|E||F|^2\ge Cq^{\frac{3d+1}{2}}$  or $|E|^2|F|\ge Cq^{\frac{3d+1}{2}}$, then we have
\[|\Delta(E)+\Delta(F)|> \frac{q}{2}.\]
\end{theorem}

To see how much Theorem \ref{thm0} is better than Theorem \ref{thm1mz}, let us make a brief comparison.

If $|E||F|\ge  Cq^{d+\frac{1}{3}}$, then either $|E||F|^2\ge Cq^{\frac{3d+1}{2}}$ or $|E|^2|F|\ge Cq^{\frac{3d+1}{2}}$. Indeed, otherwise, one has $|E||F|< C^{\frac{2}{3}}q^{d+\frac{1}{3}}$, a contradiction. When $|E|=|F|$, both Theorems \ref{thm1mz} and \ref{thm0} give the same exponent $\frac{d}{2}+\frac{1}{6}$. For $0<\epsilon<\frac{1}{3}$, one can check that for $E, F\subset \mathbb{F}_q^d$ with $|E|=q^{\frac{d-1}{2}+2\epsilon}$ and $|F|=q^{\frac{d+1}{2}-\epsilon}$, we have $|E||F|=q^{d+\epsilon}$. For these sets, Theorem \ref{thm1mz} does not tell us about the size of $\Delta(E)+\Delta(F)$, but Theorem \ref{thm0} gives the expected lower bound $cq$.

Theorem \ref{thm0} is essentially sharp in odd dimensions. To show this, let us recall the following result from \cite{hart}. If either $d\ge 2$ is even and $q\equiv 1\mod 4$ or $d=4k$, $k\in \mathbb{N}$, then there exist $\frac{d}{2}$ independent vectors $\{v_1, \ldots, v_{\frac{d}{2}}\}$ in $\mathbb{F}_q^d$ such that $v_i\cdot v_j=0$ for all $1\le i, j\le \frac{d}{2}$. Thus, if either $d=4k+3$ and $q\equiv 1\mod 4$ or $d=4k+1$, $k\in \mathbb{N}$, then we always can choose a subspace $V$ of $\frac{d-1}{2}$ vectors in $\mathbb{F}_q^{d-1}\times \{0\}$ such that $u\cdot v=0$ for all $u, v\in V$. Set $E=V$. It is not hard to check that $\Delta(E)=\{0\}$.
It has also been indicated in \cite{hart} that for any $\epsilon>0$, there exists a set $F\subset \mathbb{F}_q^d$ such that $|F|\sim q^{\frac{d+1}{2}-\epsilon}$ and $|\Delta(F)|\sim q^{1-\epsilon}$. In other words, we have $|E||F|^2\sim q^{\frac{3d+1}{2}-2\epsilon}$ and $|\Delta(E)+\Delta(F)|\sim q^{1-\epsilon}$ for any $\epsilon>0$.

It is worth noting that one can not hope to prove Theorem \ref{thm0} in all even dimensions with $q\equiv 1\mod 4$, namely, in Section \ref{prof}, we will construct examples which tell us that there are sets $E,F\subset \mathbb{F}_q^d$ with $d$ even and $q\equiv 1\mod 4$ such that $|E||F|^2\sim q^{\frac{3d}{2}+\frac{2}{3}}$ and $|\Delta(E)+\Delta(F)|\le q/2$.

While Theorem \ref{thm0} is sharp in odd dimensions, we believe that in the corresponding even dimensions the right condition should be $|E||F|^2\ge C q^\frac{3d}{2}$ or $|E|^2|F|\ge C q^\frac{3d}{2}$, which is in line with the Erd\H{o}s-Falconer distance conjecture. The difference between these cases will be seen clearly in our coming proofs.  Let us discuss shortly the main obstacle here. In our method, we will reduce the problem to making a spherical $L^2$ restriction estimate, more precisely, to finding a good upper bound of   $\max_{r\in \mathbb F_q} \sum_{m\in S_r^{d-1}} |\widehat{F}(m)|^2$ for $F\subset \mathbb{F}_q^d$, which might be attained at $r=0$ since the sphere $S_0^{d-1}$ of zero radius has the structure which is similar to that of cones whose Fourier decay is generally not very good (for example, see \cite{kohcone}). More precisely, if we rewrite
\[\sum_{m\in S_0^{d-1}} |\widehat{F}(m)|^2=q^{-d}\sum_{x, y\in F}\widehat{S_0^{d-1}}(x-y)\le q^{-d}|F|^2\cdot\max_{m\ne \mathbf{0}}|\widehat{S_0^{d-1}}(m)|\]
then, it is known that $\max_{m\ne \mathbf{0}}|\widehat{S_0^{d-1}}(m)|\sim q^{-\frac{d+1}{2}}$ for $d\ge 3$ odd, but the maximal value is close to $q^{-\frac{d}{2}}$ for $d\ge 4$ even. Therefore, to overcome this situation, one can think of two ways: either removing the contribution from $S_0^{d-1}$  or finding conditions on $q$ and $d$ in which the corresponding Fourier decay is not worse. In the former case,  
by following the proof of Theorem 3.5 in \cite{kohsun} with some slight modifications, we have $|\Delta(E)+\Delta(F)|\ge q/144$ under two conditions $|E|^2|F|\ge Cq^{\frac{3d+1}{2}}$ and $|F|^2|E|\ge Cq^{\frac{3d+1}{2}}$, for some large positive constant $C$. This result holds without any conditions on $q$ and $d$. However, it is clear that this result is weaker than both Theorems \ref{thm1mz} and \ref{thm0}. Hence, it seems that the latter will be our last chance to get some improvements. It has been recently indicated in \cite{IKSPS} that when $d=4k-2$ and $q\equiv 3\mod 4$, the zero radius sphere gives us an optimal $L^2$ restriction theorem, which is much better than that of spheres of non-zero radii. This interesting aspect comes from the fact that in the Fourier decay of $S_0^{d-1}$, the Gauss sum, whose explicit form is known, plays the crucial role. Gluing these observations together gives us Theorem \ref{thm0}. 

Restricting our attention to prime fields,  we obtain an improvement of Theorem \ref{thm0} in two dimensions. The key ingredient in our proof comes from the very recent paper on the Erd\H{o}s-Falconer distance conjecture in the plane, which says that for any set $E\subset \mathbb{F}_q^2$ with $|E|\ge q^{5/4}$, one has $|\Delta(E)|\gg q$. We refer the interested reader to \cite{mu} for more details.
\begin{theorem}\label{primecase}
Let $q\equiv 3\mod 4$ be a prime number, $E, F$ be sets in $\mathbb{F}_q^2$, and $C$ be a large positive constant. If either $|E|^4|F|^6\ge Cq^{11}$ or $|E|^6|F|^4\ge Cq^{11},$ then we have \[|\Delta(E)+\Delta(F)|\gg q.\]
\end{theorem}
The following corollary is an immediate consequence of the above theorem.
\begin{corollary}
Let $q\equiv 3\mod 4$ be a prime number, $E$ be a set in $\mathbb{F}_q^2$, and $C$ be a large positive constant. If $|E|\ge Cq^{11/10}$, then we have \[|\Delta(E)+\Delta(E)|\gg q.\]
\end{corollary}

\section{ preliminary lemmas}

In this section, we review some basics on the discrete Fourier analysis, and then use these to  derive  some lemmas  that will be used in proving  Theorem 2.1.

\subsection{Discrete Fourier analysis and Gauss sums}
The Fourier transform is defined by
$$ \widehat{f}(\alpha)=q^{-n} \sum_{\beta\in \mathbb F_q^n} \chi(-\alpha\cdot \beta) f(\beta),$$
where $f$ is a complex valued function on $\mathbb F_q^n.$
Here, and throughout this paper,  $\chi$ denotes the principle additive character of $\mathbb F_q.$  Recall that
the orthogonality of the additive character $\chi$ says that
$$ \sum_{\alpha\in \mathbb F_q^n} \chi(\beta\cdot \alpha)
=\left\{\begin{array}{ll} 0\quad &\mbox{if}\quad \beta\ne (0,\ldots, 0),\\
q^n\quad &\mbox{if}\quad \beta=(0,\ldots,0). \end{array}\right.$$
As a direct application of the orthogonality of $\chi$, the following Plancherel theorem can be proved:
 for any set $\Omega$ in $\mathbb F_q^n,$
$$ \sum_{\alpha\in \mathbb F_q^n} |\widehat{\Omega}(\alpha)|^2= q^{-n}|\Omega|.$$
In this paper,  we identify a set $\Omega$ with the indicator function $1_\Omega$ on $\Omega.$
It is not hard to prove the following formula which is referred as the Fourier inversion theorem
$$f(\beta)=\sum_{\alpha\in \mathbb F_q^n} \chi(\alpha\cdot \beta) \widehat{f}(\alpha).$$

Throughout this paper,  we denote by $\eta$  the quadratic character of $\mathbb F_q.$   For $a\in \mathbb F_q^*,$ the Gauss sum $G_a$ is defined by
$$G_a=\sum_{s\in \mathbb F_q^*}\eta(s) \chi(as).$$
The Gauss sum $G_a$ is also written as
$$ G_a=\sum_{s\in \mathbb F_q} \chi(as^2)=\eta(a) G_1.$$
The absolute value of the Gauss sum $G_a$ is exactly $q.$ Moreover, the explicit 
value of the Gauss sum $G_1$ is well--known.
\begin{lemma}[\cite{LN97}, Theorem 5.15]\label{ExplicitGauss}
Let $\mathbb F_q$ be a finite field with $ q= p^{\ell},$ where $p$ is an odd prime and $\ell \in {\mathbb N}.$
Then we have
$$G_1= \left\{\begin{array}{ll}  {(-1)}^{\ell-1} q^{\frac{1}{2}} \quad &\mbox{if} \quad p \equiv 1 \mod 4 \\
{(-1)}^{\ell-1} i^\ell q^{\frac{1}{2}} \quad &\mbox{if} \quad p\equiv 3 \mod 4.\end{array}\right. $$
\end{lemma}
The following corollary  follows from the explicit value of the  Gauss sum $G_1$. For the sake of completeness, we include a proof here. 
\begin{corollary} \label{Corm}Let $\eta$ be the quadratic character of $\mathbb F_q^*.$ Then, for any positive integer $n\equiv 2 \mod {4}$ and $q\equiv 3 \mod{4},$ we have
$$ G_1^n= -q^{\frac{n}{2}}.$$
\end{corollary}
\begin{proof} Since $q\equiv 3 \mod{4}$, Lemma \ref{ExplicitGauss} implies that
$ G_1= {(-1)}^{\ell-1} i^\ell q^{\frac{1}{2}}$ for some odd integer $\ell \ge 1.$
Since $n=4k-2$  with $k\in \mathbb N,$ we have
$$ G_1^n=  (-1)^{(\ell-1) (4k-2)}   i^{\ell (4k-2)}  q^{n/2}= (-1)^\ell q^{n/2}= -q^{n/2}.$$ \end{proof}

Completing the square and using a change of variables, it is not hard to show that
\begin{equation}\label{ComSqu}
 \sum_{s\in \mathbb F_q} \chi(as^2+bs)= \eta(a)G_1 \chi\left(\frac{b^2}{-4a}\right),\end{equation}
for any $a\in \mathbb{F}_q^*$ and $b\in \mathbb{F}_q$.

\subsection{ $L^2$ Fourier restriction estimates for spheres}
We recall that  for each $ r\in \mathbb F_q,$  the sphere $S^{d-1}_r$ in $\mathbb F_q^d$ with radius $r$ is defined by
$$ S^{d-1}_r=\{x\in \mathbb F_q^d: \sum_{i=1}^d x_i^2=r\}.$$
Notice that  compared to the Euclidean setting, we have no condition on $r$. 

It is well--known that the Fourier transform $\widehat{S^{d-1}_j}(m)$   is closely related to  the Kloosterman sum
$$K(a, b):= \sum_{s\in \mathbb F_q^*}  \chi(as+ b/s),$$
or  the twisted Kloosterman sum
$$TK(a,b):=  \sum_{s\in \mathbb F_q^*} \eta(s) \chi(as+ b/s),$$
where $a, b\in \mathbb F_q,$ and $\eta$ denotes the quadratic character of $\mathbb F_q.$ In particular, the next lemma  was given in \cite[Lemma 4]{IK10}.
\begin{lemma} \label{FTFS}  For $m\in \mathbb F_q^d,$  let $\delta_0(m)=1$  if $m=(0,\ldots, 0)$ and  $0$ otherwise.
\begin{enumerate}
\item  If $d\ge 3$ is an odd integer, then  for $m\in \mathbb F_q^d,$
$$ \widehat{S^{d-1}_j}(m)=  q^{-1} \delta_0(m) + q^{-d-1} \eta(-1)G_1^d TK\left(j, \frac{||m||}{4}\right).$$
\item  If $d\ge 2$ is an even integer, then  for $m\in \mathbb F_q^d,$
$$ \widehat{S^{d-1}_j}(m)=  q^{-1} \delta_0(m) + q^{-d-1} G_1^d  K\left(j, \frac{||m||}{4}\right).$$
\end{enumerate}
\end{lemma}
Recall that   $|K(a,b)|\le 2 q^{1/2}$ if $ab\ne 0,$ and  $|TK(a,b)|\le 2 q^{1/2}$ if $a, b\in \mathbb F_q$ (for example, see \cite{LN97}).

For $F\subset \mathbb F_q^d$  and $j\in \mathbb F_q, $ consider the  $L^2$ Fourier restriction  for the sphere $S_j^{d-1}$
$$ \mathcal{M}_j(F):=\sum_{m\in S_j^{d-1}} |\widehat{F}(m)|^2.$$
It follows from the definition of the Fourier transform that
\begin{equation}\label{MjF}\mathcal{M}_j(F) =  q^{-d} \sum_{x,y\in F} \widehat{S_j^{d-1}}(x-y).\end{equation}
In the next result, we give an upper bound for this quantity. 
\begin{proposition}\label{pro2.4} Let $F$ be a subset of $\mathbb F_q^d.$  Suppose that  either $d=4k-2$ and $q\equiv 3\mod 4$ or $d\ge 3$ is odd. Then we have
$$ \max_{j\in \mathbb F_q} \mathcal{M}_j(F) \le q^{-d-1} |F| + 2 q^{\frac{-3d-1}{2}} |F|^2.$$
\end{proposition}
\begin{proof}
{\bf Case $1$:} Suppose that  $d$ is odd.  Combining \eqref{MjF} with  the first part of  Lemma \ref{FTFS},  we get
$$ \mathcal{M}_j(F) =q^{-d-1} |F| + q^{-2d-1}G_1^d \eta(-1) \sum_{x,y\in F}  TK(j, ||x-y||/4).$$
Since the absolute value of the twisted Kloosterman sum is bounded by $2 q^{1/2}$ and $ |G_1|=q^{1/2},$    we have
$$ \mathcal{M}_j(F)  \le  q^{-d-1} |F| + 2 q^{\frac{-3d-1}{2}} |F|^2.$$
Since this bound is independent of $r\in \mathbb F_q,$   we complete the proof.\\
{\bf Case $2$:}  Assume that  $d$ is even  with $ d=4k-2,$  and   $q\equiv 3 \mod{4}.$    By combining \eqref{MjF} with  the second part of  Lemma \ref{FTFS},   we have
$$\mathcal{M}_j(F) =q^{-d-1} |F| + q^{-2d-1}G_1^d  \sum_{x,y\in F}  K(j, ||x-y||/4).$$
If $j\ne 0,$  then  the Kloosterman sum $ | K(j, ||x-y||/4) | $ is bounded by $2 q^{1/2}.$
Hence,   as in Case 1,  we obtain that
$$ \max_{t\ne 0} \mathcal{M}_j(F) \le q^{-d-1} |F| + 2 q^{\frac{-3d-1}{2}} |F|^2.$$
To complete the proof,  it therefore remains to show that
\begin{equation}\label{RT1}  \mathcal{M}_0(F) \le q^{-d-1} |F| + 2 q^{\frac{-3d-1}{2}} |F|^2.\end{equation}
In fact, we can prove that
\begin{equation}\label{RT}\mathcal{M}_0(F) \le q^{-d-1} |F| +  q^{\frac{-3d-2}{2}} |F|^2,\end{equation}
which is much stronger. Indeed,
$$\mathcal{M}_0(F)=q^{-d-1} |F| + q^{-2d-1}G_1^d  \sum_{x,y\in F}  K(0, ||x-y||/4)$$
$$=q^{-d-1} |F| + q^{-2d-1}G_1^d  \sum_{x,y\in F: ||x-y||=0}  K(0, 0) +   q^{-2d-1}G_1^d  \sum_{x,y\in F: ||x-y|| \ne0}K(0, ||x-y||/4).$$
Using the facts that $K(0,0)=q-1,$ and  $K(0, s)=-1$ for $s\ne 0,$
 \begin{align*} \mathcal{M}_0(F) &= q^{-d-1} |F|+ q^{-2d-1}G_1^d  \sum_{x,y\in F: ||x-y||=0} (q-1) +   q^{-2d-1}G_1^d  \sum_{x,y\in F: ||x-y|| \ne0} (-1)\\
&=q^{-d-1} |F| + q^{-2d} G_1^d \sum_{x, y\in F: ||x-y||=0} 1   - q^{-2d-1} G_1^d  \sum_{x,y\in F} 1. \end{align*}
Since $G_1^d=-q^{d/2}$ by Corollary \ref{Corm},  the second term above  is negative and the third term equals  $q^{\frac{-3d-2}{2}} |F|^2.$
Thus, we obtain that
$$ \mathcal{M}_0(F)\le q^{-d-1} |F|+q^{\frac{-3d-2}{2}} |F|^2,$$
as required.

We remark here that one can apply directly Theorem 1.3 in \cite{IKSPS} for characteristic functions to give a bound which is better than (\ref{RT1}), but weaker than (\ref{RT}). 
\end{proof}

\section{Proof of Theorem \ref{thm0}}\label{prof}
In this section, we devote ourselves to giving a proof of Theorem \ref{thm0}. 
Our idea is to involve a suitable algebraic variety in the  Fourier analysis. An advantage in using an algebraic variety argument is that it offers a new form for the  upper bound of the $L^2$-norm of a certain counting function, which is more manageable for our purpose afterwards.

\subsection{Algebraic variety and related Fourier transform}
  Let $X=(\mathbf{x},\mathbf{y})$ be the coordinates of  $ \mathbb F_q^{2d} \times \mathbb F_q^{2d}=\mathbb F_q^{4d},$ and let $ ||X||_*$ be the homogeneous polynomial defined by
$$ ||X||_*:= ||\mathbf{x}||-||\mathbf{z}||=x_1^2+\cdots + x_{2d}^2-z_1^2-\cdots -z_{2d}^2.$$

\begin{definition} \label{defV0} Let $V_0$ be the subvariety of $\mathbb F_q^{4d}$ cut out by the equation $ ||X||_*=0$, i.e.,
$$ V_0:=\{\ X\in \mathbb F_q^{4d}: ||X||_*=0\ \}.$$
\end{definition}
We need the following Fourier transform of the  variety $V_0$ in $\mathbb F_q^{4d}.$
\begin{lemma} \label{defVFT} If $M\in \mathbb F_q^{4d},$  then we have
$$ \widehat{V_0}(M):=q^{-4d} \sum_{X\in V_0} \chi(-M\cdot X)=\left\{\begin{array}{ll} q^{-1}\delta_0(M) +  q^{-2d-1} (q-1) \quad &\mbox{if}\quad||M||_*=0,\\
   -  q^{-2d-1}\quad &\mbox{if} \quad||M||_*\ne 0. \end{array} \right.$$                -
\end{lemma}
\begin{proof}
It follows from the orthogonality of $\chi$ that
$$ \widehat{V_0}(M)=q^{-4d} \sum_{X\in V_0} \chi(-M\cdot X)=q^{-1} \delta_0(M) + q^{-4d-1} \sum_{X\in \mathbb F_q^{4d}} \sum_{s\ne 0} \chi(s||X||_* - M\cdot X).$$
From the formula \eqref{ComSqu},  it follows
$$\widehat{V_0}(M)=q^{-1}\delta_0(M) + q^{-4d-1} {G}_1^{4d} \sum_{s\ne 0} \eta^{2d}(s) \eta^{2d}(-s) \chi\left(\frac{||M||_*}{-4s}\right).$$
Since $\eta^{2d}=1,$ by a change of variables, one has
$$\widehat{V_0}(M)=q^{-1}\delta_0(M) +  q^{-4d-1} {G}_1^{4d} \sum_{r\ne 0}\chi(r||M||_*).$$
By the orthogonality of $\chi$,
$$ \widehat{V_0}(M)=\left\{\begin{array}{ll} q^{-1}\delta_0(M) +{G}_1^{4d}  q^{-4d-1} (q-1) \quad &\mbox{if}\quad||M||_*=0,\\
   -{G}_1^{4d}  q^{-4d-1}\quad &\mbox{if} \quad||M||_*\ne 0. \end{array} \right.$$
Since $G_1^{4d}= q^{2d},$  the proof is complete.
\end{proof}

By invoking Lemma \ref{defVFT},  we are able to deduce the following lemma.
\begin{lemma}  \label{goodLem}
Let $\mathcal{D} \subset \mathbb F_q^{2d}.$  For each $t\in \mathbb{F}_q$, let $\nu(t)$ be the number of pairs $(\mathbf{x}, \mathbf{y})\in \mathcal{D}\times \mathcal{D}$ such that $||\mathbf{x}-\mathbf{y}||=t$. Then we have
$$  \sum_{t\in \mathbb F_q} \nu(t)^2 \le \frac{|\mathcal{D}|^4}{q}  +  q^{6d}\sum_{||M||_*=0} |\widehat{\mathcal{D}\times \mathcal{D}}(M)|^2.$$
\end{lemma}
\begin{proof}
Since $\nu(t)=\sum_{\mathbf{x}, \mathbf{y}\in \mathcal{D}: ||\mathbf{x}-\mathbf{y}||=t} 1,$ we have
$$ \sum_{t\in \mathbb F_q} \nu(t)^2=\sum_{t\in \mathbb F_q} \left(\sum_{\mathbf{x}, \mathbf{y}\in \mathcal{D}: ||\mathbf{x}-\mathbf{y}||=t} 1   \right)^2 = \sum_{\mathbf{x,y,z,w} \in \mathcal{D}: ||\mathbf{x}-\mathbf{y}||=||\mathbf{z}-\mathbf{w}||} 1.$$
We will relate  the value $\sum_{t} \nu(t)^2$  to  the Fourier transform  on the variety $V_0$  in $\mathbb F_q^{4d}.$
To do this,  we let $X=(\mathbf{x}, \mathbf{z}),  Y=(\mathbf{y}, \mathbf{w}) \in \mathcal{D}\times \mathcal{D}. $  Using these notations with $|| \cdot ||_*$,  we can write
$$ \sum_{t\in \mathbb F_q} \nu(t)^2=  \sum_{X, Y\in \mathcal{D} \times \mathcal{D} : ||X-Y||_*=0}   1= \sum_{X, Y\in \mathcal{D} \times \mathcal{D} }  V_0(X-Y),$$
where  we recall from Definition \ref{defV0} that  the  variety $V_0$  is given by
$$ V_0= \{ X\in \mathbb F_q^{4d}:  ||X||_*=0\}.$$
Applying the Fourier inversion theorem to the characteristic function $V_0(X-Y),$ we get
$$ \sum_{t\in \mathbb F_q} \nu(t)^2=\sum_{X,Y\in \mathcal{D}\times \mathcal{D}} V_0(X-Y)= q^{8d} \sum_{M\in \mathbb F_q^{4d}} \widehat{V_0}(M) |\widehat{\mathcal{D}\times \mathcal{D}}(M)|^2.$$
Replacing $\widehat{V_0}(M)$ by the explicit value given  in  Lemma \ref{defVFT},   we get
$$ \sum_{t\in \mathbb F_q} \nu(t)^2= q^{8d-1} \sum_{||M||_*=0} \delta_0(M) |\widehat{\mathcal{D}\times \mathcal{D}}(M)|^2 + q^{6d}\sum_{||M||_*=0} |\widehat{\mathcal{D}\times \mathcal{D}}(M)|^2 - q^{6d-1} \sum_{M\in \mathbb F_q^{4d}} |\widehat{\mathcal{D}\times \mathcal{D}}(M)|^2. $$
Since the last term of the RHS  is negative and  the first term   is  $|\mathcal{D}|^4/q,$     we  have
$$ \sum_{t\in \mathbb F_q} \nu(t)^2 \le \frac{|\mathcal{D}|^4}{q}  +  q^{6d}\sum_{||M||_*=0} |\widehat{\mathcal{D}\times \mathcal{D}}(M)|^2,$$
as desired.
\end{proof}

We are ready to prove Theorem \ref{thm0}.
\bigskip

\paragraph{\textbf{Proof of Theorem \ref{thm0}}:}
Define $\mathcal{D}=E\times F\subset \mathbb F_q^d \times \mathbb F_q^d$. For $t\in \mathbb{F}_q$, let $\nu(t)$ be the number of pairs $(\mathbf{x}, \mathbf{y})\in \mathcal{D}\times \mathcal{D}$ such that $||\mathbf{x}-\mathbf{y}||=t$.
By the Cauchy-Schwarz inequality, we have
\begin{equation}\label{formudis}  |\Delta(E)+\Delta(F)|=|\Delta(E\times F)|\ge \frac{|E|^4|F|^4}{\sum\limits_{t\in  \mathbb F_q}\nu(t)^2}.\end{equation}
Since $|\mathcal{D}|=|E||F|,$   Lemma \ref{goodLem} implies that
$$\sum_{t\in \mathbb F_q} \nu(t)^2 \le  \frac{|E|^4|F|^4}{q} + q^{6d} \sum_{M'\in \mathbb F_q^{3d}}  |\widehat{E\times F \times E} (M')|^2  \max_{r\in \mathbb F_q} \sum_{m\in S_r^{d-1}} |\widehat{F}(m)|^2. $$
Using Proposition \ref{pro2.4} and the Plancherel theorem, we get
$$\sum_{t\in \mathbb F_q} \nu(t)^2\le \frac{|E|^4|F|^4}{q} +  q^{6d} \left(q^{-3d} |E|^2|F|\right) \left( q^{-d-1} |F| + 2 q^{\frac{-3d-1}{2}} |F|^2\right).$$
Simplifying  the RHS gives us
$$ \sum_{t\in \mathbb F_q} \nu(t)^2 \le   \frac{|E|^4|F|^4}{q} + q^{2d-1}|E|^2|F|^2 + 2 q^{\frac{3d-1}{2}} |E|^2|F|^3.$$
This estimate can be combined with \eqref{formudis} to deduce that
$$ |\Delta(E)+\Delta(F)| >q/2$$
under the conditions that  $|E||F|\ge 2 q^d$ and $|E|^2|F| \ge C q^{(3d+1)/2}$ for a sufficiently large constant $C.$
Now we show that  the condition $|E||F|\ge 2q^d$ is not necessary.
Recall that  we  can assume that  $|E|< 4 q^{(d+1)/2}$, otherwise   $|\Delta(E)|=q,$  which is  the consequence due to Iosevich and Rudnev \cite{io}.
We claim  that  if $|E|^2|F| \ge C q^{(3d+1)/2}$, then $|E||F|\ge 2 q^d.$
If not, then  $|E||F|< 2 q^d$  and  $|E|^2|F| \ge C q^{(3d+1)/2}.$ These two conditions clearly imply that $|E|> \frac{C}{2} q^{(d+1)/2},$ which contradicts  our assumption that $|E| < 4 q^{(d+1)/2}.$

A symmetric argument by switching the roles of $E$ and $F$  also yields that
if $|E||F|^2 \ge C q^{(3d+1)/2},$ then $|\Delta(E) + \Delta(F)|> q/2.$  This completes the proof. $\square$

In the following construction, we show that Theorem \ref{thm0} can not hold when $d\ge 2$ is even and $q\equiv 1\mod 4$.
\begin{construction}\label{sharp-odd} Let $d\ge 2$ is even.
\begin{enumerate}
\item Suppose $q=p^l$ with $p\equiv 1\mod 4$ and $l=3k$. There exist sets $E, F\subset \mathbb{F}_q^d$ such that $|E||F|^2\sim q^{d+\frac{2}{3}}$, and 
$|\Delta(E)+\Delta(F)|\le q/2.$
\item Suppose that $q=p^l$ with $p\equiv 3\mod 4$ and $l=6k$. There exist sets $E, F\subset \mathbb{F}_q^d$ such that $|E||F|^2\sim q^{d+\frac{2}{3}}$, and 
$|\Delta(E)+\Delta(F)|\le q/2.$
\end{enumerate}

\end{construction}
\begin{proof}
We first recall the following result from \cite{mur} due to Murphy and Petridis. 
\begin{itemize}
\item If $q=p^l$ with $p\equiv 1\mod 4$ and $l=3k$, then there exists a set $A\subset \mathbb{F}_q^2$ such that $|A|\sim q^{4/3}$ and $|\Delta(A)|\le q/2$.
\item If $q=p^l$ with $p\equiv 3\mod 4$ and $l=6k$, then there exists a set $A\subset \mathbb{F}_q^2$ such that $|A|\sim q^{4/3}$ and
$|\Delta(A)|\le q/2$.
\end{itemize}
We note that $q=p^l\equiv 3\mod 4$ if and only if $p\equiv 3\mod 4$ and $l$ is odd. 

Since $d\ge 2$ is even and $q\equiv 1\mod 4$, it is known in \cite{hart} that one can find $\frac{d}{2}$ independent vectors in $\mathbb{F}_q^d$, say $\{v_1, \ldots, v_{\frac{d}{2}}\}$, such that $v_i\cdot v_j=0$ for all $1\le i, j\le \frac{d}{2}$, and $\frac{d-2}{2}$ independent vectors in $\mathbb{F}_q^{d-2}$, say $\{v_1', \ldots, v_{\frac{d-2}{2}}'\}$, such that $v_i'\cdot v_j'=0$ for all $1\le i, j\le \frac{d-2}{2}$. Define $E=\mathtt{Span}(v_1, \ldots, v_{\frac{d}{2}})$ and $F=\mathtt{Span}(v_1', \ldots, v_{\frac{d-2}{2}}')\times A$. It is clear that $|E|=q^{d/2}$ and $|F|\sim q^{\frac{d-2}{2}+\frac{4}{3}}$. So, $|E||F|^2\sim q^{\frac{3d}{2}+\frac{2}{3}}$. It follows from definitions of $E$ and $F$ that $\Delta(E)=\{0\}$ and $\Delta(F)=\Delta(A)$. Thus, $|\Delta(E)+\Delta(F)|=|\Delta(A)|\le q/2$.
\end{proof}

\section{Proof of Theorem \ref{primecase}}

In this section,  we give a proof of Theorem \ref{primecase}. To do this, we first recall the following proposition from \cite{pham}, which essentially says that the $L^2$-norm of the distance measure on the Cartesian product set $E\times F$ can be reduced to the $L^2$-norm of the distance measure on each component.

\begin{proposition}\label{proLove} Let $E, F \subset \mathbb F_q^d.$ For any $r\in \mathbb{F}_q$, let $\nu(r)$ be the number of pairs $((e_1, f_1), (e_2, f_2))\in (E\times F)^2$ such that $||e_1-e_2||+||f_1-f_2||=r$, and let $\mu(r)$ be the number of pairs $(x, y)\in E\times E$ such that $||x-y||=r$. Then we have
 \[\sum_{r\in \mathbb F_q}\nu(r)^2\le \frac{|E|^4|F|^4}{q}+q^d|F|^2\sum_{r\in \mathbb F_q}\mu(r)^2.\]
\end{proposition}
\paragraph{\textbf{Proof of Theorem \ref{primecase}}:}
If $|E|\ge q^{5/4}$, then it has been proved in \cite[Theorem 1]{mu} that $|\Delta(E)|\gg q$. Therefore, without loss of generality, we can assume that $|E|\le q^{5/4}$.

Let $T(E)$ be the number of triples $(x, y, z)\in E\times E\times E$ such that $||x-y||=||x-z||$ with $||y-z||\ne 0$. Since $q\equiv 3\mod 4$, there are no two distinct points $y, z\in E$ such that $||y-z||=0$.
It has been proved in \cite[Theorem 4]{mu} that there exists a large enough constant $C_2$ such that
\[T(E)\le C_2\left(\frac{|E|^3}{q}+q^{2/3}|E|^{5/3}+q^{1/4}|E|^2\right)\ll q^{2/3}|E|^{5/3},\]
for  $ |E|\le q^{5/4}$.
By the Cauchy-Schwarz inequality, we have
\[\sum_{r\in \mathbb F_q}\mu(r)^2\le |E|\cdot (T(E)+|E|^2) \le C_2\left( \frac{|E|^4}{q}+q^{2/3}|E|^{8/3}+q^{1/4}|E|^3\right)+|E|^3\ll q^{2/3}|E|^{8/3},\]
for $|E|\le q^{5/4}$.
By Proposition \ref{proLove},  it follows that
\[\sum_{r\in \mathbb F_q}\nu(r)^2\le \frac{|E|^4|F|^4}{q}+q^2|F|^2\sum_{r\in \mathbb F_q}\mu(r)^2.\]
Therefore,
\[\sum_{r\in \mathbb F_q}\nu(r)^2\ll \frac{|E|^4|F|^4}{q}+q^\frac{8}{3}|F|^2|E|^{\frac{8}{3}}\ll \frac{|E|^4|F|^4}{q},\]
whenever $|E|^4|F|^6\gg q^{11}$.

As in the proof of Theorem \ref{thm0}, we have
\[|\Delta(E)+\Delta(F)|\ge \frac{|E|^4|F|^4}{\sum_{r}\nu(r)^2}\gg q,\]
under the condition $|E|^4|F|^6\gg q^{11}$.

We also change the roles of $E$ and $F$  in the above proof. Hence, we also see that
if $|E|^6|F|^4 \gg q^{11},$ then $|\Delta(E) + \Delta(F)|\gg q/2.$  This completes the proof. $\square$

\section*{Acknowledgments}
Daewoong Cheong and Doowon Koh were supported by Basic Science Research Programs through National Research Foundation of Korea (NRF) funded by the Ministry of Education (NRF-2018R1D1A3B07045594 and NRF-2018R1D1A1B07044469, respectively). Thang Pham was supported by Swiss National Science Foundation grant P400P2-183916. The authors would like to thank Vietnam Institute for Advanced Study in Mathematics for hospitality during their visit.

\end{document}